\documentclass[12pt]{article}
\usepackage{hyperref,amsmath,amssymb,amscd,amsthm,makeidx,txfonts,graphicx,url,bm,relsize}
\usepackage{a4wide,xy}
\usepackage{float,blkarray}
\usepackage{tikz}

\let\oldmarginpar\marginpar
\renewcommand\marginpar[1]{\-\oldmarginpar[\raggedleft\footnotesize #1]
{\raggedright\footnotesize #1}}

\makeindex

\xyoption{all}
\input{xypic}

\numberwithin{equation}{section}

\newtheorem{theorem}{Theorem}[subsection]
\newtheorem{proposition}[theorem]{Proposition}

\newtheorem{lemma}[theorem]{Lemma}

\theoremstyle{remark}
\newtheorem{remark}[theorem]{Remark}

\theoremstyle{definition}

\newtheorem{nothing}[subsubsection]{}

\def\SL{{\rm SL}}
\def\PGL{{\rm PGL}}
\def\GL{{\rm GL}}
\def\calG{{\mathcal{G}}}
\def\calT{{\mathcal{T}}}

\def\B{{\rm B}}

\DeclareMathOperator{\Tr}{Tr}

\def\calF{{\mathcal{F}}}

\def\F{\mathbb{F}}

\def\Q{\overline{\mathbb{Q}}_\ell}

\def\t{{\rm t}}
\def\calC{{\mathcal C}}

\def\gl{{\rm gl}}

\begin{document}

\title{Fourier transform from the symmetric square representation of $\PGL_2$ and $\SL_2$}

\author{ G\'erard Laumon
\\ {\it Universit\'e Paris-Saclay, LMO, CNRS}
\\{\tt gerard.laumon@u-psud.fr } \and Emmanuel Letellier \\ {\it
  Universit\'e Paris Cit\'e, IMJ-PRG} \\{\tt
 emmanuel.letellier@imj-prg.fr}
 }

\pagestyle{myheadings}

\maketitle

\begin{abstract}Let $G$ be a connected reductive group over $\overline{\F}_q$ and let $\rho^\vee:G^\vee\rightarrow\GL_n$ be an algebraic  representation of the dual group $G^\vee$. Assuming that $G$ and $\rho^\vee$ are defined over $\F_q$, Braverman and Kazhdan defined an operator on the space $\calC(G(\F_q))$ of complex valued functions on $G(\F_q)$. In this paper we are interested in the case where $G$ is either $\SL_2$ or $\PGL_2$ and $\rho^\vee$ is the symmetric square representation of $G^\vee$. We construct a natural $G\times G$-equivariant embedding $G\hookrightarrow\calG=\calG_\rho$ and an involutive operator (Fourier transform) $\calF^{\calG}$ on the space of functions $\calC(\calG(\F_q))$ that extends Braverman-Kazhdan's operator.

  \end{abstract}
\tableofcontents

\section{Introduction}\label{intro}

\subsection{Motivations}

\noindent  Let $\mathbb{A}$ be the ring of adeles of a global field $F$ and let $\psi:\mathbb{A}/F\rightarrow\mathbb{C}^\times$ be a non-trivial additive character. We can define the Fourier transform of a locally constant function $f:\gl_n(\mathbb{A})\rightarrow \mathbb{C}$ with compact support as

$$
\calF^{\gl_n}(f)(x)=\int_{\gl_n(\mathbb{A})}\psi(\Tr(xy))f(y)dy.
$$
Tate, for $n=1$ \cite{Tate}, and Godement-Jacquet for arbitrary $n$ \cite{GJ} proved that this Fourier transform provides functional equations between $L$-functions.  The \emph{involutivity property} of $\calF^{\gl_n}$, namely

$$
\calF^{\gl_n}\circ\calF^{\gl_n}(f)=f^-
$$
where $f^-(x):=f(-x)$,  is an important ingredient.
\bigskip

Given now a connected reductive group $G$ together with a representation $\rho^\vee:G^\vee\rightarrow\GL_n$, Braverman-Kazhdan \cite{BK1}, Lafforgue \cite{La} and more recently Ng\^o \cite{Ngo}, considered a Fourier operator  on a certain space of functions on $G(\mathbb{A})$ defined as a restricted product of Fourier operators over local fields $F_v$. These local Fourier operators are defined spectrally and are not involutive. One of the challenge would be to extend them to (involutive) Fourier transforms in order to extend Godement-Jacquet's work to $L$-functions defined from $\rho^\vee$. 
\bigskip

The construction of the Fourier operator relatively to $\rho^\vee$ and the problem of extending them to involutive Fourier transforms make also sense when replacing local fields by finite fields. The aim of this paper is to construct such a Fourier transform (in the finite field case)  when $\rho^\vee$ is the symmetric power representation of $\SL_2$ or $\PGL_2$. We now explain this in details.

\subsection{Braverman-Kazhdan-Fourier operators on $G$}

Let $G$ be a connected reductive algebraic group with an $\F_q$-structure and  associated geometric Frobenius $F:G\rightarrow G$. In this paper, all geometric Frobenius with be denoted by $F$ and relative to $q$. Denote by $G^\vee$ the dual group of $G$ (in the sense of Deligne-Lusztig), a connected reductive algebraic group defined over $\F_q$. Lusztig has defined a partition of the set $\widehat{G^F}$ of all irreducible $\Q$-characters of $G^F$ (with $\ell$ a prime not dividing $q$) whose parts, called \emph{Lusztig series},  are indexed by the set of $F$-stable semisimple conjugacy classes of $G^\vee$.  We then denote by ${\rm LS}(G)$ the set of Lusztig series of $G^F$.
\bigskip

Let  us equip $\GL_n$ with its standard Frobenius and assume given a morphism $\rho^\vee:G^\vee\rightarrow\GL_n$ which commutes with Frobenius. It induces a map between sets of $F$-stable semisimple conjugacy classes and so a transfert map 

$$
\t_\rho:{\rm LS}(G)\rightarrow{\rm LS}(\GL_n)
$$
between Lusztig series (see \cite[\S 3.6]{LL} for more details). This is the analogue of Langlands functoriality for finite fields.
\bigskip

For a finite set $X$, let $\calC(X)$ denotes the $\Q$-space of all functions $X\rightarrow\Q$. 
\bigskip

{\bf Standard Fourier operator on $\calC(\GL_n^F)$}
\bigskip

\noindent Fix a non-trivial additive character $\psi:\F_q\rightarrow\Q^\times$ and denote by $\Tr$ the trace map on matrices. We have a natural Fourier transform $\calF^{\gl_n}:\calC(\gl_n^F)\rightarrow\calC(\gl_n^F)$ defined by

$$
\calF^{\gl_n}(f)(y)=\sum_{x\in\gl_n^F}\psi(\Tr(yx))f(x),
$$
for any $f\in\calC(\gl_n^F)$. The involutivity property of $\calF^{\gl_n}$ reads

$$
\calF^{\gl_n}\circ\calF^{\gl_n}(f)=q^{n^2}\, f^-
$$
where $f^-(x):=f(-x)$.
\bigskip

The above Fourier transform restricts to an operator 
$$
\calF^{\GL_n}:\calC(\GL_n^F)\rightarrow\calC(\GL_n^F)
$$
(which unlike $\calF^{\gl_n}$ is not involutive) defined by

$$
\calF^{\GL_n}(f)(h)=\sum_{g\in \GL_n^F}\psi(\Tr(hg))f(g).
$$
The above operator defines a function $\gamma^{\GL_n}:\widehat{\GL_n^F}\rightarrow\Q$ which is constant on Lusztig series and which satisfies

$$
\calF^{\GL_n}(\chi)=\gamma^{\GL_n}(\chi)\,\chi^*
$$
for any irreducible character $\chi$ of $\GL_n^F$ where $\chi^*$ denotes the dual character of $\chi$ (see \cite[\S 3.1]{LL} for more details).
\bigskip

{\bf BK-Fourier operator on $\calC(G^F)$}
\bigskip

\noindent  We transfert the gamma function $\gamma^{\GL_n}$ to a gamma function $\gamma^G_\rho:\widehat{G^F}\rightarrow\Q$ by the formula

\begin{equation}
\gamma^G_\rho:=c_Gc_{\GL_n}^{-1}\,\gamma^{\GL_n}\circ\t_\rho,
\label{gamma}\end{equation}
with

$$
c_G:=(-1)^{\F_q-\text{rank}(G)}q^{\nu(G)}
$$
where $\nu(G)$ is the number of positive roots of $G$.
\bigskip

Following Braverman-Kazhdan \cite{BK}, define a central function $\phi^G_\rho:G^F\rightarrow\Q$, called BK-Fourier kernel, by the formula

$$
\phi^G_\rho:=\sum_{\chi\in\widehat{G^F}}\chi(1)\gamma^G_{\rho}(\chi)\,\chi^*.
$$
The BK-Fourier operator $\calF^G_\rho:\calC(G^F)\rightarrow\calC(G^F)$ is defined by 

$$
\calF^G_\rho(f)(h)=\sum_{g\in G^F}\phi^G_\rho(hg)f(g).
$$
Notice that 

$$
\calF^G_\rho(\chi)=\gamma^G_\rho(\chi)\, \chi^*
$$
for any irreducible character $\chi$ (see \cite[\S 3.5]{LL} for more details).
\bigskip

This is the spectral definition of $\calF^G_\rho$ due to Braverman-Kazhdan. They also gave an explicit  conjectural formula for $\phi^G_\rho$ which they prove for $G=\GL_m$ and which is proved in full generality in \cite{LL} (see also \cite{NC} \cite{C} for a more geometrical proof).

\subsection{The problem}

\noindent A crucial property of Fourier transforms is the involutivity property. While $\calF^{\GL_n}$ arises as the restriction of the Fourier transform $\calF^{\gl_n}$, we can ask ourselves wether the operator $\calF^G_\rho$ is also the restriction of an involutive operator. More precisely, 
\bigskip

{\bf Question :} is there an open $G\times G$-equivariant embedding  $G\hookrightarrow \calG_\rho$ defined over $\F_q$ (where $G\times G$ acts by right and left multiplication on $G$) and an involutive operator $\calC(\calG_\rho^F)\rightarrow\calC(\calG_\rho^F)$ restricting to $\calF^G_\rho$ on functions on $G^F$ ?
\bigskip

This question has a positive answer in some simple cases. For instance,  if $\rho^\vee:\SL_n\hookrightarrow\GL_n$ is the inclusion then there exists a dual morphism $\rho:\GL_n\rightarrow\PGL_n$ which is the natural quotient \cite[\S 3.5]{LL}.  The BK-Fourier kernel $\phi^G_\rho$ is then just $\rho_!(\psi\circ\Tr)$.  We can then take the quotient stack $[\gl_n/{\rm Ker}(\rho)]=[\gl_n/\GL_1]$ as $\calG_\rho$ and consider on it  the kernel $\overline{\rho}_!(\psi\circ \Tr)$ where $\overline{\rho}:\gl_n\rightarrow[\gl_n/\GL_1]$ is the quotient. This defines an involutive Fourier transform on the space of functions $\calC(\calG_\rho^F)$ which extends $\calF_\rho^G$.
\bigskip

We will see in this paper, that in fact $\calG_\rho$ may not be not  unique (see \S \ref{geounique}).
\bigskip

\subsection{The main results}

Here we mention only the results in odd characteristic (the characteristic $2$ case is much simpler and is treated independently in \S \ref{char2}). Let $G$ be either $\SL_2$ or $\PGL_2$ equipped with the standard Frobenius $F$ and consider the symmetric square representation $\rho^\vee:G^\vee\rightarrow\GL_3$ (see \S \ref{construction}).

We will use the following identification for $G$

\begin{align*}
&G=\{(x,\alpha,b)\in\GL_2\times\GL_1\times\GL_1\,|\,\det(x)=\alpha^2\}/\GL_1\times\GL_1&\text{ if } G=\SL_2\\
&G=\{(x,\alpha,b)\in\GL_2\times\GL_1\times\GL_1\,|\,\det(x)=\alpha^2\}/\GL_1\times\mu_2\times\GL_1&\text{ if } G=\PGL_2
\end{align*}
where $\mu_2$ acts on the second coordinate by $\alpha\mapsto -\alpha$ and $\GL_1\times\GL_1$ acts by $(s,t)\cdot (x,\alpha,b)=(sx,s\alpha,tb)$.
\bigskip

Consider the stack $\calG=\calG_\rho$ defined as follows

\begin{align*}
&\calG:=[\{(x,\alpha,b)\in\gl_2\times\gl_1\times\gl_1\,|\,\det(x)=\alpha^2\}/\GL_1\times\GL_1]&\text{ if } G=\SL_2\\
&\calG:=[\{(x,\alpha,b)\in\gl_2\times\gl_1\times\gl_1\,|\,\det(x)=\alpha^2\}/\GL_1\times\mu_2\times\GL_1]&\text{ if } G=\PGL_2\\
\end{align*}

Notice that the inclusion $i:G\hookrightarrow \calG$ is a $G\times G$-equivariant open embedding (where $G\times G$ acts by left and right multiplication).

\bigskip

Consider the kernel $\phi^\calG:\calG^F\rightarrow\Q$ defined as follows (see \S \ref{mr}):
\bigskip

\noindent If $G=\SL_2$, 

\begin{align*}
\phi^\calG([x,\alpha,b])&=\sum_{s,t \in\F_q^\times}\psi\left(s(\Tr(x)+2\alpha)+tb\right)\\
&=1-q\delta_{\Tr(x)+2\alpha=0}.
\end{align*}
If $G=\PGL_2$,

$$
\phi^\calG([x,\alpha,b])=\begin{cases}\sum_{s,t \in\F_q^\times,\varepsilon\in\mu_2}\psi\left(s(\Tr(x)+2\varepsilon \alpha)+tb\right)&\text{ if }\alpha\in\gl_1^F,\\
0&\text{ if }\alpha\in\gl_1^{F'}.\end{cases}
$$
where $F':\gl_1\rightarrow\gl_1, x\mapsto -x^q$.
\bigskip

Our first main theorem is the following one (see \S \ref{mr}).

\begin{theorem} (i) The Fourier transform $\calF^\calG:\calC(\calG^F)\rightarrow\calC(\calG^F)$ defined from the kernel $\phi^\calG$ is \emph{involutive}, more precisely

$$
\calF^\calG\circ\calF^\calG(f)=q^5 f
$$
for any $f\in\calC(\calG^F)$.

\noindent (ii) $\calF^\calG$ extends the BK-Fourier operator $\calF^G=\calF^G_\rho$, i.e.

$$
\calF^G=i^*\circ\calF^\calG\circ i_!.
$$

\end{theorem}

To prove the above theorem we need the following unsatisfactory trick

$$
\sum_{s\in\F_q^\times} \psi(s\tau)=\sum_{s\in\F_q^\times}\psi(s\tau^2)
$$
as an essential ingredient.
\bigskip

Finally we will discuss the case of  the symmetric square representation of $G=\GL_2$. The BK-Fourier kernel turns out be  more complicated. However there is still a natural stack $\calG_\rho$ and a natural Fourier operator on the space of functions on $\F_q$-points but it is not involutive. 

\section{Braverman-Kazhdan kernels}

We choose a finite field $\F_q$ and a prime $\ell$ which does not divide $q$. In sections \ref{SL2case} \ref{PGL2case} \ref{GL} the characteristic of $\F_q$ will be assumed to be odd. Throughout this paper $\alpha_o:\F_q^\times\rightarrow\Q^\times$ will denote the character that takes the value $1$ at squares and the value $-1$ at non-squares. We also fix a non-trivial additive linear character $\psi:\F_q\rightarrow\Q^\times$ and denote by $\Tr$ the trace map on matrices. The maximal torus of $\GL_n$ of diagonal matrices is denoted by $T_n$.

\subsection{Generalities for tori}\label{gentori}

Let $S$, $S^\vee$ be two tori defined over $\F_q$ with geometric Frobenius $F$. Recall that $(S,F)$ and $(S^\vee,F)$ are dual if there exists an isomorphism $X(S)\simeq Y(S^\vee)$ compatible with Galois action where $X(S)$ denotes the character group of $S$ and $Y(S^\vee)$ the co-character group of $S^\vee$.
\bigskip

Given a morphism $\rho^\vee:S^\vee\rightarrow S'{^\vee}$ defined over $\F_q$, we get a \emph{dual} morphism $\rho:S'\rightarrow S$ (see for instance \cite[\S 3.5]{LL}). More concretely, if 
$$
\rho^\vee:S^\vee\rightarrow T_n^\vee=T_n,\hspace{.5cm}s\mapsto (\alpha_1(s),\dots,\alpha_n(s))
$$
with $\alpha_1,\dots,\alpha_n\in X(S^\vee)$, then 

$$
\rho: T_n\rightarrow S, (t_1,\dots,t_n)\mapsto \alpha_1(t_1)\cdots\alpha_n(t_n)
$$
where $\alpha_1,\dots,\alpha_n$ are regarded as co-characters of $S$ under the isomorphism $X(S^\vee)\simeq Y(S)$.
\bigskip

If $T_n$ is equipped with the standard Frobenius that raises matrix coefficients to their $q$-th power, the BK-Fourier kernel $\phi_\rho^S$ is then (see \cite[Lemma 5.2.1]{LL})

$$
\phi_\rho^S=c_Sc_{T_n}^{-1}\rho^F_!(\psi\circ\Tr|_{T_n}).
$$
where 
\begin{equation}
c_S:=(-1)^{\F_q-\text{rank}(S)}.
\end{equation}

\subsection{The symmetric square representations}\label{construction}

Let $G$ be denote either $\PGL_2$, $\SL_2$  or $\GL_2$ over $\overline{\F}_q$ equipped with standard Frobenius $F:G\rightarrow G$ that raises matrix coefficients to their $q$-th power. Denote by $(G^\vee,F)$ the dual of $(G,F)$ (in the sense of Deligne-Lusztig) so that $\PGL_2^\vee=\SL_2$, $\SL_2^\vee=\PGL_2$ and $\GL_2^\vee=\GL_2$ equipped with standard Frobenius.

Denote by 

$$
\rho^\vee:G^\vee\rightarrow \GL_3
$$
 the \emph{symmetric square representation} of $G^\vee$, namely if we let ${\rm Sym}^2:\GL_2\rightarrow\GL_3$ be defined by
 
  $$
{\rm Sym}^2 \left(\begin{array}{cc}a&b\\c&d\end{array}\right)=\left(\begin{array}{ccc}a^2&ab&b^2\\
 2ac&ad+bc&2bd\\c^2&cd&d^2\end{array}\right),
 $$
 then 
 
 \begin{align*}
 \rho^\vee&={\rm Sym}^2\text{ if }G=\GL_2\\
 \rho^\vee&={\rm Sym}^2|_{\SL_2}\text{ if } G=\PGL2\\
 \rho^\vee&={\rm Sym}^2\cdot {\det}^{-1}\text{ if } G=\SL_2.
 \end{align*} 
 
\begin{remark} Notice that we have factorizations

\begin{equation}
\xymatrix{\SL_2\ar[r]^{\pi^\vee}\ar[rd]_{f^\vee}&\GL_2\ar[r]^{{\rm Sym}^2}&\GL_3\\
&\PGL_2\ar[ru]_{{\rm Sym}^2\cdot\det^{-1}}&}
\label{triangle}\end{equation}
where $f^\vee$ is the canonical quotient and $\pi^\vee$ the canonical inclusion.
\label{rem1}\end{remark}

The morphism $\rho^\vee$ commutes thus  with the Frobenius morphisms. We denote by $T$ the maximal torus of $G$  consisting of diagonal matrices, and we denote by $W=\{1,\sigma\}\simeq S_2$ the Weyl group of $G$ with respect to $T$. The morphism $\rho^\vee$ induces a morphism $W\rightarrow S_3$ mapping the non-trivial element $\sigma$ of $W$ to the transposition $(13)\in S_3$ which we also denote by $\sigma$.
  
We choose an element $g\in G$ (resp. $g\in\GL_3$) such that 

$$
g^{-1}F(g)=\sigma.
$$
We then denote by $T_\sigma$ (resp. $T_{3,\sigma}$) the $F$-stable maximal torus $gTg^{-1}$ (resp. $gT_3g^{-1}$). 

Under conjugation by $g^{-1}$ we have
$$
(T_\sigma,F)\simeq (T,\sigma F),\hspace{1cm}(T_{3,\sigma},F)\simeq (T_3,\sigma F).
$$
In particular, $T_\sigma^F\simeq T^{\sigma F}$ and $T_{3,\sigma}^F\simeq T_3^{\sigma F}$. It will be more convenient to represent elements of $T_\sigma^F$ in $T^{\sigma F}$. 

The morphism $\rho^\vee$ restricts to a $W$-equivariant morphism $T^\vee\rightarrow T_3$ inducing (see \cite[Section 3.5]{LL}) a $W$-equivariant morphism $\rho:T_3\rightarrow T$ given as follows

\begin{align*}
\rho(a,b,c)=(a^2b,bc^2)&\hspace{.5cm}\text{ if }G=\GL_2\\
\rho(a,b,c)=(a/c,c/a)&\hspace{.5cm}\text{ if } G=\SL_2\\
\rho(a,b,c)=[a^2b,bc^2]&\hspace{.5cm}\text{ if } G=\PGL_2
\end{align*}
where $[a^2b,bc^2]$ denotes the class of $(a^2b,bc^2)$ in $\PGL_2$.

We thus get morphisms $\rho^F:T_3^F\rightarrow T^F$ and $\rho^{\sigma F}:T_3^{\sigma F}\rightarrow T^{\sigma F}$, and the Braverman-Kazhdan kernels (BK-Fourier kernels) $\phi^T$ and $\phi^{T_\sigma}$ on $T^F$ and $T_\sigma^F$ respectively are given by (see \S \ref{gentori})

\begin{equation}
\phi^T=c_{T_3}c_T^{-1}\rho^F_!(\psi\circ\Tr),\hspace{1cm}\phi^{T_\sigma}=c_{T_{3,\sigma}}c_{T_\sigma}^{-1}\,\rho^{\sigma F}_!(\psi\circ\Tr).
\label{BKT}\end{equation}
(we omit $\rho$ from the notation of the BK-Fourier kernels as the context will be always clear). The BK-Fourier kernel on $G^F$ is then given by the formula \cite[Theorem 5.2.3]{LL}
\begin{equation}
\phi^G=\frac{1}{2}\left(R_T^G(\phi^T)+R_{T_\sigma}^G(\phi^{T_\sigma})\right),
\label{LL}\end{equation}
where $R_T^G$ and $R_{T_\sigma}^G$ denote Deligne-Lusztig induction.

The right hand side of (\ref{LL}) can be computed explicitly. For $G=\GL_2$ we get : 

\begin{scriptsize}
\begin{equation}
\label{general}
\begin{array}{|c|c|c|c|c|c|}
\hline
&&&&\\
&g=\left(\begin{array}{cc}a&0\\0&a\end{array}\right)&g=\left(\begin{array}{cc}a&0\\0&b\end{array}\right) &  g=\left(\begin{array}{cc}x&0\\0& x^q\end{array}\right)&g=\left(\begin{array}{cc}a&1\\0&a\end{array}\right)\\
&a\in\F_q^\times&a\neq b\in\F_q^\times&x\neq x^q\in\F_{q^2}^\times&a\in\F_q^\times\\
\hline
&&&&\\
\phi^G&\frac{1}{2}\left((q+1)\phi^T(g)+(1-q)\phi^{T_\sigma}(g)\right)&\phi^T(g)&\phi^{T_\sigma}(g)&\frac{1}{2}\left(\phi^T(g_s)+\phi^{T_\sigma}(g_s)\right)\\
&&&&\\
\hline
\end{array}
\end{equation}
\end{scriptsize}
where $g_s$ denotes the semisimple part of $g$.

\begin{remark}Notice that the map $f^\vee:\SL_2\rightarrow\PGL_2$ in (\ref{triangle}) is a \emph{normal} morphism (i.e. its image is a normal subgroup) with dual morphism $f=f^\vee:\SL_2\rightarrow\PGL_2$ (see \cite[\S 3.5]{LL}). From the transitivity of the transfert map between Lusztig series (see \cite[\S 3.6]{LL}), we get from \cite[Lemma 4.1.1]{LL} that

\begin{equation}
\phi^{\PGL_2}=\begin{cases}f^F_!(\phi^{\SL_2})\\
-q^{-1}\pi_!^F(\phi^{\GL_2}).\end{cases}
\label{f}\end{equation}
Notice that $c_{\PGL_2}c_{\GL_2}^{-1}=-1$ and $c_{\PGL_2}c_{\SL_2}^{-1}=1$ (see definition of the gamma function (\ref{gamma})).
\end{remark}

\subsection{Explicit computation : The case $G=\SL_2$}\label{SL2case}
\bigskip

\noindent Recall that in this case

\begin{align*}
&\rho^F:T_3^F\rightarrow T^F,\hspace{.2cm} (a,b,c)\mapsto (a/c,c/a),\\&\rho^{\sigma F}:T_3^{\sigma F}\rightarrow T^{\sigma F},\hspace{.2cm} (a',b,a'{^q})\mapsto (a'{^{1-q}},a'{^{q-1}}).
\end{align*}
Notice that $c_{T_3}c_T^{-1}=c_{T_{3,\sigma}}c_{T_\sigma}^{-1}=1$ and 

$$
{\rm Ker}(\rho^F)=\{(s,t,s)\,|\, s,t\in\F_q^\times\}={\rm Ker}(\rho^{\sigma F}).
$$
Therefore

\begin{align*}
\phi^T(a/c,c/a)&=\rho^F_!(\psi\circ\Tr)(a/c,c/a)\\
&=\sum_{s,t}\psi(s(a+c)+bt)\\
&=-\sum_s\psi(s(a+c))\\
\phi^{T_\sigma}(a'{^{1-q}},a'{^{q-1}})&=\rho^{\sigma F}_!(\psi\circ\Tr)(a'{^{1-q}},a'{^{q-1}})\\
&=\sum_{s,t}\psi(s(a'+a'{^q})+bt)\\
&=-\sum_s\psi(s(a'+a'{^q}))
\end{align*}
where $a'\in\F_{q^2}^\times$.
\bigskip

We now use the identification

$$
\SL_2\simeq\{(x,\alpha)\in\GL_2\times\GL_1\,|\, \det(x)=\alpha^2\}/\GL_1
$$
where $\GL_1$ acts by $s\cdot(x,\alpha)=(sx,s\alpha)$, and we write the elements of $\SL_2$ in the form $[x,\alpha]$ with $\det(x)=\alpha^2$.

Notice that

$$
\SL_2^F=\{(x,\alpha)\in\GL_2^F\times\GL_1^F\,|\, \det(x)=\alpha^2\}/\GL_1^F.
$$

The morphism $\rho^F$ and $\rho^{\sigma F}$ becomes

$$
\rho^F(a,b,c)=[(a^2,c^2),ac],\hspace{.5cm}\rho^{\sigma F}(a',b,a'{^q})=[(a'{^2},a'{^{2q}}),a'{^{q+1}}].
$$
The  crucial ingredient is the following obvious identity

$$
\sum_{s\in \F_q^\times}\psi(s\tau)=\sum_{s\in\F_q^\times}\psi(s\tau^2)=\begin{cases}q-1 &\text{ if }\tau=0,\\ -1 &\text{otherwise.}\end{cases}
$$
Indeed we get
$$
\phi^T([(a^2,c^2),ac])=-\sum_{s\in\F_q^\times}\psi(s(a+c))=-\sum_{s\in\F_q^\times}\psi(s(a+c)^2)=-\sum_{s\in\F_q^\times}\psi(s(a^2+c^2+2ac)).
$$
We thus have for $[x,\alpha]\in T^F$

$$
\phi^T([x,\alpha])=-\sum_{s\in\F_q^\times}\psi\left(s(\Tr(x)+2\alpha)\right),
$$
and similarly for $[x',\alpha]\in T_\sigma ^F$, we have

$$
\phi^{T_\sigma}([x',\alpha])=-\sum_{s\in\F_q^\times}\psi\left(s(\Tr(x')+2\alpha)\right).
$$
We thus deduce from Formula (\ref{LL}) the following one.

\begin{proposition}For $[x,\alpha]\in\SL_2^F$ we have

$$
\phi^G([x,\alpha])=-\sum_{s\in\F_q^\times}\psi\left(s(\Tr(x)+2\alpha)\right)=\begin{cases}-(q-1)&\text{ if }\Tr(x)+2\alpha=0,\\1&\text{ otherwise.}\end{cases}
$$
\label{BSL}\end{proposition}

\begin{remark}It will be useful to get rid of the sign by writing $\SL_2$ in the form

$$
\SL_2\simeq\left[\left\{(x,\alpha,b)\in\GL_2\times\GL_1\times\GL_1\,|\, \det(x)=\alpha\right\}/\GL_1\times\GL_1\right]
$$
where $\GL_1\times\GL_1$ acts as $(s,t)\cdot(x,\alpha,b)=(sx,s\alpha,tb)$. Indeed, under this identification

$$
\phi^G([x,\alpha,b])=\sum_{s,t\in\F_q^\times}\psi\left(s(\Tr(x)+2\alpha)+bt\right).
$$
\label{remSL}

\end{remark}

\subsection{Explicit computation : The case $G=\PGL_2$}\label{PGL2case}

Recall that

$$
\rho:T_3\rightarrow T,\hspace{1cm}(a,b,c)\mapsto [a^2b,bc^2]=[a/c,c/a].
$$
and so 

$$
{\rm Ker}(\rho^F)=\{(s,t,\varepsilon s)\,|\, s,t\in\F_q^\times, \varepsilon\in\mu_2\},\hspace{.5cm}
{\rm Ker}(\rho^{\sigma F})=\{(u,t,u^q)\,|\, t\in\F_q^\times, u\in\F_{q^2}^\times,  u^q=\pm u\}.
$$
We thus have

\begin{align*}\phi^T([a/c,c/a])&=\sum_{s,t\in\F_q^\times,\,\varepsilon\in\mu_2}\psi(s(a+\varepsilon c)+t)\\
&=-\sum_{s,\varepsilon}\psi(s(a+\epsilon c)^2)\\
&=-\sum_{s,\varepsilon}\psi(s(a^2+c^2+2\varepsilon ac))\\
\phi^{T_\sigma}([a'{^{1-q}},a'{^{q-1}}])&=-\sum_{s\in\F_q^\times,\,\varepsilon\in\mu_2}\psi(s(a'{^2}+a'{^{2q}}+2\varepsilon a'{^{q+1}})).
\end{align*}
\bigskip

Let us identify 

$$\PGL_2=\SL_2/\mu_2=\left\{(x,\alpha)\in\GL_2\times\GL_1\,|\, \det(x)=\alpha^2\right\}/\GL_1\times\mu_2
$$
where $\mu_2$ acts by multiplication on the second coordinate. Notice that

$$
\PGL_2^F=\left\{(x,\alpha)\in\GL_2^F\times(\GL_1^F\sqcup\GL_1^{F'})\,|\, \det(x)=\alpha^2\right\}/\GL_1^F\times\mu_2
$$
where $F':\GL_1\rightarrow\GL_1, x\mapsto -x^q$ is the standard Frobenius twisted by the non-trivial element of $H^1(F,\mu_2)=\mu_2$.

Using Formula (\ref{LL}) we deduce the following one.

\begin{proposition}For $[x,\alpha]\in\PGL_2^F$ we have

$$
\phi^G([x,\alpha])=\begin{cases}-\sum_{s\in\F_q^\times,\,\varepsilon\in\mu_2}\psi(s(\Tr(x)+2\varepsilon\alpha))&\text{ if }\alpha\in\GL_1^F,\\
0 & \text{ if }\alpha\in\GL_1^{F'}.\end{cases}
$$
\label{BPGL}\end{proposition}

Notice that the result could have been also deduced from Formula (\ref{f}) and Proposition \ref{BSL}.

\begin{remark}From Proposition \ref{BSL} and Proposition \ref{BPGL}, we see that the kernel $\phi^G$ descends to the GIT quotient $T/\!/W$ when $G$ is $\SL_2$ or $\PGL_2$. We will see in the next section that this is not true when $G=\GL_2$.
\end{remark} 

\subsection{Explicit computation : The case $G=\GL_2$}\label{GL}

We have

$$
\rho:T_3\rightarrow T,\hspace{1cm}(a,b,c)\mapsto (a^2b,bc^2).
$$
and so

$$
{\rm Ker}(\rho^F)=\{(s,s^{-2},\varepsilon s)\,|\, s\in\F_q^\times,\varepsilon\in\mu_2\},\hspace{.5cm}{\rm Ker}(\rho^{\sigma F})=\{(s,s^{-2},s^q)\,|\, s^q=\pm s, s\neq 0\}.
$$
For $b\in\F_q$, put
$$
\kappa(b):=\sum_{s\in\F_q^\times}\psi(s^{-2}b) \hspace{1cm}\text{ and}\hspace{1cm}\kappa'(b):=\sum_{\{s\neq 0,\, s^q=-s\}}\psi(s^{-2}b).
$$
Notice that if $r\in\F_q^\times$ is a non-square, then 

$$
\kappa(rb)=\kappa'(b).
$$
Since $c_{T_3,T}=c_{T_{3,\sigma},T_\sigma}=-1$, we have (for $a,b,c\in\F_q^\times$)

\begin{align*}
\phi^T(a^2b,bc^2)&=-\sum_{s\in\F_q^\times,\,\varepsilon\in\mu_2}\psi(s(a+\varepsilon c)+s^{-2}b)\\
&=\begin{cases}-\sum_{s,\varepsilon}\psi(s+s^{-2}(a+\varepsilon c)^2b)&\text{ if } a^2\neq c^2,\\
-\left(\kappa(b)+\sum_s\psi(s+s^{-2}(a+\varepsilon c)^2b)\right)&\text{ if }a=\varepsilon c,\end{cases}
\end{align*}
using the variable change $s(a+\varepsilon c)\leftrightarrow s$ when $a+\varepsilon c\neq 0$. 
\bigskip

We get the formula for $\phi^{T_\sigma}$ as follows (where $s_o\in\F_{q^2}^\times$ satisfies $s_o^q=-s_o$)

\begin{align*}
\phi^{T_\sigma}(a'{^2}b,ba'{^{2q}})&=-\sum_{s^q=s}\psi(s(a'+a'{^q})+s^{-2}b)-\sum_{s^q=-s}\psi(s(a'-a'{^q})+s^{-2}b)\\
&=-\sum_{s^q=s}\psi(s(a'+a'{^q})+s^{-2}b)-\sum_{s^q=s}\psi(ss_o(a'-a'{^q})+s_o^{-2}s^{-2}b)\\
&=\begin{cases}-\sum_{s\in\F_q^\times,\,\varepsilon\in\mu_2}\psi\left(s+s^{-2}(a'+\epsilon a'{^q})^2b\right)&\text{ if }a'{^q}\neq \pm a',\\
-\sum_{s\in\F_q^\times}\psi\left(s+s^{-2}(a'+a'{^q})^2b\right)-\kappa'(b)&\text{ if }a'{^q}= a',\\
-\sum_{s\in\F_q^\times}\psi\left(s+s^{-2}(a'-a'{^q})^2b\right)-\kappa(b)&\text{ if }a'{^q}=-a'.
\end{cases}
\end{align*}
using the variable changes $ss_o(a'-a'{^q})\leftrightarrow s$ if $a'{^q}\neq a'$ and $s(a'+a'{^q})\leftrightarrow s$ if $a'{^q}\neq -a'$.

\begin{remark} The case $a'{^q}=-a'$ is not essential as we can write $(a'{^2}b,ba'{^{2q}})$ in the form $(a^2d,da^2)$ with $a,d\in\F_q^\times$ by putting $(a,d):=(s_oa',s_o^{-2}b)$.
\end{remark}

We now consider the identification

\begin{equation}
\GL_2\simeq \{(x,\alpha,b)\in \GL_2\times\GL_1\times\GL_1\,|\, \det(x)=\alpha^2\}/\GL_1\times\mu_2
\label{idGL}\end{equation}
where $\GL_1$ acts as $t\cdot[x,\alpha,b]=[tx,t\alpha,t^{-1}b]$ and $\mu_2$ acts by multiplication on the second coordinate (under this identification, the element $[x,\alpha,b]$ corresponds to $xb\in\GL_2$).

We have

$$
\GL_2^F=\left\{(x,\alpha,b)\in \GL_2^F\times(\GL_1^F\sqcup\GL_1^{F'})\times\GL_1^F\,|\, \det(x)=\alpha^2\right\}/\GL_1^F\times\mu_2.
$$

In the following table we write the elements of $\GL_2^F$ in the form $[x,\alpha,b]$ and we put

$$
\Phi([x,\alpha,b]):=\sum_{s\in\F_q^\times,\,\varepsilon\in\mu_2}\psi\left(s+s^{-2}b(\Tr(x)+2\varepsilon \alpha)\right).
$$

\begin{scriptsize}
\begin{equation}
\label{tableGL}
\begin{array}{|c|c|c|c|}
\hline
&&&\\
&\left[\left(\begin{array}{cc}a^2&0\\0&a^2\end{array}\right),a^2,b\right]&\left[\left(\begin{array}{cc}a^2&0\\0&c^2\end{array}\right),ac,b\right] &  \left[\left(\begin{array}{cc}a'{^2}&0\\0& a'{^{2q}}\end{array}\right), a'{^{q+1}},b\right]\\
&a,b\in\F_q^\times&a,c\in\F_q^\times, a^2\neq c^2&a'{^2}\neq a'{^{2q}}\\
\hline
&&&\\
\phi^T&-\kappa(b)-1-\Phi([x,\alpha,b])&-\Phi([x,\alpha,b])&\times \\
&&&\\
\hline
&&&\\
\phi^{T_\sigma}&-\kappa'(b)-1-\Phi([x,\alpha,b])&\times&-\Phi([x,\alpha,b])\\
&&&\\
\hline
\end{array}
\end{equation}
\end{scriptsize}

Notice that for $b\neq 0$

$$
\kappa(b)+\kappa'(b)=-2,\hspace{1cm}\frac{1}{2}(\kappa(b)-\kappa'(b))=\sum_{t\in\F_q^\times}\alpha_o(t)\psi(tb)=:S(\alpha_o,\psi_b)
$$
where recall that $\alpha_o$ is the non-trivial square root of the identity character of $\F_q^\times$ (it takes the value $1$ at squares and $-1$ at non-squares).

Using Table (\ref{general}) and the above table we get the following proposition.
\bigskip

\begin{proposition} For $[x,\alpha,b]\in\GL_2^F$ we have $\phi^G([x,\alpha,b])=0$ if $\alpha\in\GL_1^{F'}$ and otherwise we have

\begin{scriptsize}
\begin{equation}
\label{table}
\begin{array}{|c|c|c|c|c|c|}
\hline
&&&&\\
&\left[\left(\begin{array}{cc}a^2&0\\0&a^2\end{array}\right),a^2,b\right]&\left[\left(\begin{array}{cc}a^2&0\\0&c^2\end{array}\right),ac,b\right] &  \left[\left(\begin{array}{cc}a'{^2}&0\\0& a'{^{2q}}\end{array}\right), a'{^{q+1}},b\right]&\left[\left(\begin{array}{cc}a^2&1\\0&a^2\end{array}\right),a^2,b\right]\\
&a,b\in\F_q^\times&a,c\in\F_q^\times, a^2\neq c^2&a'{^2}\neq a'{^{2q}}&a,b\in\F_q^\times\\
\hline
&&&&\\
\phi^G&-qS(\alpha_o,\psi_b)-\Phi([x,\alpha,b])&-\Phi([x,\alpha,b])&-\Phi([x,\alpha,b])&-\Phi([x,\alpha,b])\\
&&&&\\
\hline
\end{array}
\end{equation}
\end{scriptsize}
This can be re-written for $g\in\GL_2^F$ as

$$
\phi^G(g)=\begin{cases}0&\text{ if } \det(g)\notin(\F_q^\times)^2,\\
-\sum_{s,\varepsilon}\psi\left(s+s^{-2}\left(\Tr(g)+2\varepsilon \sqrt{\det(g)}\right)\right)&\text{ if }\det(g)\in(\F_q^\times)^2\text{ and } g\text{ not central},\\
-\sum_{s,\varepsilon}\psi\left(s+s^{-2}\left(\Tr(g)+2\varepsilon \sqrt{\det(g)}\right)\right)-q S(\alpha_o,\psi_\lambda)&\text{ if }g=\lambda\cdot I_2.
\end{cases}
$$
Notice that in the last case, $\phi^G(g)=1-\sum_s\psi(s+4s^{-2}\lambda)
-qS(\alpha_o,\psi_\lambda)$.
\end{proposition}

We see that $\phi^G$ does not descend to $T/\!/W$ as the values of $\phi^G$ depends on the unipotent part.

\section{Extending the BK-Fourier kernel: The case $G=\SL_2,\PGL_2$}

\subsection{Preliminaries}\label{Involutive}

By a finite groupo\"id, we shall mean a groupo\"id with a finite number of isomorphism classes and whose  automorphism groups are finite. For a finite groupo\" id $X$ we denote by $\overline{X}$ the set of isomorphism classes of $X$.  By a $\Q$-valued function $X\rightarrow\Q$ on $X$ we shall mean a $\Q$-valued function on $\overline{X}$ and we denote by $\calC(X)$ the $\Q$-vector space of functions on $X$. Given a morphism $\varphi:X\rightarrow Y$ of finite groupo\"ids, we have the operator $\varphi^*:\calC(Y)\rightarrow\calC(X)$ defined by composition with $\varphi$ and the operator $\varphi_!:\calC(X)\rightarrow\calC(Y)$ defined by

$$
\varphi_!(f)(y)=\sum_{x\in X_y}\frac{1}{|{\rm Aut}(x)|}\,y_X^*(f)(x)
$$
where $\bullet$ is a point, $X_y=\bullet\times_Y X$,

$$
\xymatrix{X_y\ar[d]\ar[rr]^{y_X}&&X\ar[d]^{\varphi}\\
\bullet\ar[rr]^{y}&&Y}
$$
If $G$ is a finite group acting on a finite set $X$, we denote by $[X/G]$ the groupo\"id of $G$-equivariant maps $G\rightarrow X$ and by $\pi:X\rightarrow[X/G]$ the quotient map. Then for $f\in\calC(X)$ we have

$$
\pi_!(f)(\pi(x))=\sum_{g\in G}f(g\cdot x)
$$
for any $x\in X$.
\bigskip

Given a function $K\in\calC(X\times X)$ we can define an operator

$$
\calF:\calC(X)\rightarrow\calC(X)
$$
with kernel $K$ by

$$
\calF(f)(y)=\sum_{x\in\overline{X}}\frac{1}{|{\rm Aut}(x)|}\, K(y,x)f(x)
$$
for $f\in\calC(X)$ and $y\in\overline{X}$.
\bigskip

We have the following straightforward lemma.

\begin{lemma}The operator $\calF$ is involutive, i.e. $\calF\circ\calF=1$, if and only if for all $y,z\in\overline{X}$ we have

\begin{equation}
\Delta(z,x):=\sum_{y\in\overline{X}}\frac{1}{|{\rm Aut}(y)|}\, K(z,y)K(y,x)=\begin{cases}|{\rm Aut}(z)|&\text{ if } z=x,\\
0&\text{ otherwise.}\end{cases}
\label{inv}\end{equation}
\label{involutive}\end{lemma}

Let $Z$ be a finite set endowed with a right action of a finite group $H$.

\begin{lemma}Let $K^Z\in\calC(Z\times Z)$ by $H$-invariant for the diagonal action of $H$ on $Z\times Z$ and define $K^{[Z/H]}\in\calC([Z/H]\times[Z/H])$ by

$$
K^{[Z/H]}(\overline{x},\overline{y})=\sum_{h\in H}K^Z(h\cdot x,y)=\sum_{h\in H}K^Z(x,h\cdot y)
$$
where $x,y\in Z$ maps to $\overline{x}\in\overline{[Z/H]}$ and $\overline{y}\in\overline{[Z/H]}$ respectively. If $K^Z$ satisfies (\ref{inv}) then so does $K^{[Z/H]}$.
\label{connected}\end{lemma}

\begin{proof}Suppose that $K^Z$ satisfies (\ref{inv}).

\begin{align*}
\sum_{\overline{y}\in\overline{[Z/H]}}\frac{1}{|C_H(y)|}K^{[Z/H]}(\overline{x},\overline{y})K^{[Z/H]}(\overline{y},\overline{z})&=\sum_{\overline{y}\in\overline{[Z/H]}}\frac{1}{|C_H(y)|}\sum_{h\in H}K^Z(h\cdot x,y)\sum_{h'\in H}K^Z(y,h'\cdot z)\\
&=\frac{1}{|H|}\sum_{y\in Z}\sum_{h\in H}K^Z(h\cdot x,y)\sum_{h'\in H}K^Z(y,h'\cdot z)\\
&=\frac{1}{|H|}\sum_{h,h'\in H}\sum_{y\in Z}K^Z(h\cdot x,y)K^Z(y,h'\cdot z)\\
&=\frac{1}{|H|}\sum_{h,h'\in H}\delta_{hx,h'x}\\
&=|C_H(x)|\,\delta_{\overline{x},\overline{z}}.
\end{align*}
\end{proof}

Let $X$ be an $\overline{\F}_q$-scheme on which an $\overline{\F}_q$-algebraic group $G$ acts on the right. We denote by $[X/G]$ the associated quotient stack. We assume that $X$, $G$ and the action of $G$ on $X$ are defined over $\F_q$ and we denote by $F$ the associated geometric Frobenius on $X$, $G$ and $[X/G]$. 
\bigskip

Recall that the isomorphism classes of $G$-torsors over ${\rm Spec}(\F_q)$ are parametrized by the set $H^1(F,G)=H^1(F,G/G^o)$ of $F$-conjugacy classes on $G$ and so the groupoid $[X/G]^F$ of $\F_q$-points of $[X/G]$ decomposes as

$$
[X/G]^F=\coprod_{\overline{h}\in H^1(F,G)}[X^{F\circ h}/G^{F\circ h}]
$$
where $h\in G$ denotes a representative of $\overline{h}$ and $G^{F\circ h}=\{g\in G\,|\, F(h^{-1}gh)=g\}$.
\bigskip

If $G$ is connected then $H^1(F,G)$ is trivial and so $[X/G]^F=[X^F/G^F]$.
\bigskip

Let 

$$
\pi:[X/G^o]\rightarrow [X/G]
$$
be the natural $G/G^o$-torsor.

Notice that the induced map

$$
\pi^F:[X/G^o]^F\rightarrow [X/G]^F
$$
is not surjective. But, for any $h\in G$, it induces a surjective map

$$
[X/G^o]^{F\circ h}=[X^{F\circ h}/G^{o(F\circ h)}]\rightarrow [X^{F\circ h}/G^{F\circ h}]
$$ 
with fibers isomorphic to $(G/G^o)^{F\circ h}=G^{F\circ h}/G^{o(F\circ h)}$.
\bigskip

Let $S:[X/G^o]\times [X/G^o]\rightarrow\overline{\F}_q$  be a $G/G^o$-invariant algebraic map for the diagonal action of $G/G^o$. Then
$$
S(F(y),F(x))=S(F\circ h(y),F\circ h(x))=S(y,x)^q
$$
for all $x,y\in [X/G^o]$ and $h\in G$, and so $S$ induces maps

$$
S^{F\circ h}:[X/G^o]^{F\circ h}\times [X/G^o]^{F\circ h}\rightarrow\F_q.
$$

For all $\overline{h}\in H^1(F,G)$, we get a kernel

$$
K^{[X/G^o]}_{\overline{h}}:[X/G^o]^{F\circ h}\times[X/G^o]^{F\circ h}\rightarrow \Q,\hspace{1cm}(x,y)\mapsto\psi(S(x,y)).
$$
and so a kernel

$$
K^{[X/G]}_{\overline{h}}:[X^{F\circ h}/G^{F\circ h}]\times[X^{F\circ h}/G^{F\circ h}]\rightarrow\Q,
$$
defined by

$$
K^{[X/G]}_{\overline{h}}(\overline{x},\overline{y})=\sum_{w\in(G/G^o)^{F\circ h}}\psi(S(w\cdot x,y)),
$$
for $x,y\in[X/G^o]^{F\circ h}$ with images denoted respectively by $\overline{x},\overline{y}$ in the quotient $[X^{F\circ h}/G^{F\circ h}]$.

We define a kernel

$$
K^{[X/G]}:[X/G]^F\times[X/G]^F\rightarrow\Q
$$
by 

$$
K^{[X/G]}(\overline{x},\overline{y})=\begin{cases}K^{[X/G]}_{\overline{h}}(\overline{x},\overline{y})&\text{ if }\overline{x},\overline{y}\text{ are both in } [X^{F\circ h}/G^{F\circ h}]\text{ for some }\overline{h}\in H^1(F,G),\\ 0&\text{ otherwise}.\end{cases}
$$

\begin{proposition} The Fourier transform $\calF^{[X/G]}:\calC([X/G]^F)\rightarrow\calC([X/G]^F)$ defined from the kernel $K^{[X/G]}$ is involutive if and only if all the kernels  $K_{\overline{h}}^{[X/G^o]}$, with $\overline{h}\in H^1(F,G)$, satisfy (\ref{inv}).
\label{descent}\end{proposition}

\begin{remark}The above construction of the kernel $K^{[X/G]}$ has the following geometrical counterpart. The geometric analogue of $\psi$ is the Artin-Schreier sheaf $\mathcal{L}_\psi$ on $\mathbb{A}^1$ and so $K^{[X/G^o]}$ corresponds to 

$$
\mathcal{K}^{[X/G^o]}:=S^*(\mathcal{L}_\psi).
$$
The sheaf $\mathcal{K}^{[X/G^o]}$ is $G/G^o$-equivariant for the diagonal action of $G/G^o$ (see for instance \cite[\S 2.3]{LL0} for the definition of equivariance). 

We consider
$$
\xymatrix{
[X/G^o]\times[X/G^o]\ar[rr]^{\pi\times1}&& [X/G]\times[X/G^o]\ar[rr]^{1\times \pi}&& [X/G]\times[X/G]
}
$$
Then the complex $(\pi\times 1)_!\mathcal{K}^{[X/G^o]}$ is naturally $G/G^o$-equivariant and so by \cite[Lemma 2.3.4]{LL} it descends to a complex $\mathcal{K}^{[X/G]}$ on $[X/G]\times[X/G]$, i.e. $\mathcal{K}^{[X/G]}$ is the unique complex (up to a unique isomorphism) such that

$$
(1\times\pi)^*\mathcal{K}^{[X/G]}=(\pi\times 1)_!\mathcal{K}^{[X/G^o]}.
$$
Then $\mathcal{K}^{[X/G]}$ is the geometrical counterpart of $K^{[X/G]}$, i.e. if we introduce Frobenius action, then $K^{[X/G]}$ is the characteristic function of $\mathcal{K}^{[X/G]}$ equipped with its natural Frobenius (see \cite[\S 6.1]{LL}).
\bigskip

Then the geometrical counterpart of Proposition \ref{descent} says that the geometric Fourier transform $\calF^{[X/G^o]}$ defined from $\mathcal{K}^{[X/G^o]}$ is involutive if and only if the geometric Fourier transform $\calF^{[X/G]}$ defined from $\mathcal{K}^{[X/G]}$ is.
\end{remark}

\subsection{The tori case}\label{toricase}

Let $S$ be a torus defined over $\F_q$ and assume given a surjective morphism

$$
\rho:T_n\rightarrow S
$$
defined over $\F_q$ (the morphism $\rho$ may not be surjective on $\F_q$-points).
\bigskip

Consider the BK-Fourier kernel

$$
\phi_\rho^S=\rho^F_!(\psi\circ\Tr|_{T_n})
$$
and consider the Fourier operator $\calF_\rho^S:\calC(S^F)\rightarrow\calC(S^F)$ defined from $\phi^S_\rho$.
\bigskip

We wish to extend this operator to an involutive Fourier transform.

Put

$$
H:={\rm Ker}(\rho),\hspace{1cm}\t_n:={\rm Lie}(T_n).
$$
Consider the natural open embedding

$$
\iota: S\simeq T_n/{\rm Ker}(\rho)\hookrightarrow [\t_n/H].
$$
On $[\t_n/H]^F$ define the kernel

$$
\phi^{[\t_n/H]}:=q^F_!(\psi\circ\Tr)
$$
where $q:\t_n\rightarrow[\t_n/H]$. 

Concretely

$$
\phi^{[\t_n/K]}(y)=\begin{cases}\sum_{t\in H^F}\psi(\Tr(tx))&\text{ if }y=q^F(x),\\
0& \text{ if } x\notin {\rm Im}(q^F).\end{cases}
$$
Notice that $q^F$ is not surjective if $H$ is not connected.
\bigskip

The products $\t_n\times\t_n\rightarrow\t_n$ and $H\times H\rightarrow H$ induce a map 

$$
m:[\t_n/H]\times[\t_n/H]\rightarrow[\t_n/H]
$$
which is compatible with $\F_q$-structures. We thus get a morphism of groupo\"ids

$$
m^F:[\t_n/H]^F\times[\t_n/H]^F\rightarrow[\t_n/H]^F.
$$
We then define a two-variable kernel $K^{[\t_n/H]}:[\t_n/H]^F\times[\t_n/H]^F\rightarrow\Q$ by

$$
K^{[\t_n/H]}=(m^F)^*(\phi^{[\t_n/H]}).
$$
\begin{proposition}(1) The Fourier transform $\calF^{[\t_n/H]}:\calC([\t_n/H]^F)\rightarrow\calC([\t_n/H]^F)$ defined from the kernel $K^{[\t_n/H]}$ is involutive, i.e.

$$
\calF^{[\t_n/H]}\circ\calF^{[\t_n/H]}(f)=q^n \, f^-
$$
for any $f\in\calC([\t_n/H]^F)$ where $f^-\in\calC([\t_n/H]^F)$ is defined as $f^-([x]):=f([-x])$.

\noindent (2) It extends the Fourier operator $\calF_\rho^S:\calC(S^F)\rightarrow\calC(S^F)$, i.e.

$$
\calF_\rho^S=\iota^*\circ\calF^{[\t_n/H]}\circ\iota_!.
$$
\end{proposition}

\begin{proof}The assertion (1) follows from the discussion in \S \ref{Involutive} as the Fourier transform on $\calC(\t_n^F)$ defined from the kernel $\psi\circ \Tr$ is involutive. The second assertion is straightforward.
\end{proof}

\subsection{Main result}\label{mr}

In this section the characteristic of the base field $\F_q$ is assumed to be odd, $G$ is $\SL_2$ or $\PGL_2$ and  $\rho^\vee:G\rightarrow \GL_3$ is as in \S \ref{construction}. We consider the stack $\calG$ defined as follows :

\begin{align*}
&\calG:=[\{(x,\alpha,b)\in\gl_2\times\gl_1\times\gl_1\,|\,\det(x)=\alpha^2\}/\GL_1\times\GL_1]&\text{ if } G=\SL_2,\\
&\calG:=[\{(x,\alpha,b)\in\gl_2\times\gl_1\times\gl_1\,|\,\det(x)=\alpha^2\}/\GL_1\times\mu_2\times\GL_1]&\text{ if } G=\PGL_2,\\
\end{align*}
where $\mu_2$ acts on the second coordinate and $\GL_1\times\GL_1$ as $(s,t)\cdot(x,\alpha)=(sx,s\alpha,tb)$. 
\bigskip

The natural inclusion $G\hookrightarrow\calG, g\mapsto [g,1,1]$ is open and $G\times G$-equivariant for the left and right multiplication of $G$ on $\calG$. 
\bigskip

We have

\begin{align*}
&\calG^F=[\{(x,\alpha,b)\in\gl_2^F\times\gl_1^F\times\gl_1^F\,|\,\det(x)=\alpha^2\}/\GL_1^F\times\GL_1^F]&\text{ if } G=\SL_2,\\
&\calG^F=[\{(x,\alpha,b)\in\gl_2^F\times(\gl_1^F\sqcup\gl_1^{F'})\times\gl_1\,|\,\det(x)=\alpha^2\}/\GL_1\times\GL_1\times\mu_2]&\text{ if } G=\PGL_2,\\
\end{align*}
where $F':\gl_1\rightarrow\gl_1$, $x\mapsto -x^q$. 
\bigskip

\noindent We extend the kernel $\phi^G$ to a kernel $\phi^\calG$ on $\calG^F$ as follows (see Remark \ref{remSL}).
\bigskip

If $G=\SL_2$, define $\phi^\calG$ by

$$
\phi^\calG([x,\alpha,b])=\sum_{s,t\in\F_q^\times}\psi(s(\Tr(x)+2\alpha)+tb).
$$
 
If $G=\PGL_2$, define $\phi^\calG$ by

$$
\phi^\calG([x,\alpha,b])=\begin{cases}\sum_{s,t\in\F_q^\times,\varepsilon\in\mu_2}\psi(s(\Tr(x)+2\varepsilon\alpha)+tb)&\text{ if }\alpha\in\gl_1^F,\\
0&\text{ if }\alpha\in\gl_1^{F'}.\end{cases}
$$

Consider 
$$
\calF^\calG:\calC(\calG^F)\rightarrow\calC(\calG^F)
$$
defined by

$$
\calF^\calG(f)([x',\alpha',b'])=\sum_{[x,\alpha,b]\in\calG^F}\frac{1}{|{\rm Aut}([x,\alpha,b])|}\phi^\calG([x'x,\alpha'\alpha,b'b])f([x,\alpha,b])
$$
for any function $f\in\calC(\calG^F)$.
Here is the main result of our paper.

\begin{theorem} The Fourier transform $\calF^\calG$ is \emph{involutive}, namely

$$
\calF^\calG\circ\calF^\calG(f)=q^5\, f
$$
for any $f\in\calC(\calG^F)$.
\label{mainresult}\end{theorem}

\subsection{Fourier transforms on quadratic spaces}

Assume given a vector space $V$ of dimension $n=2m+1$ over $\F_q$ and a non-degenerate quadratic form $Q:V\rightarrow\F_q$ with associated bilinear form

$$
B(v',v)=Q(v'+v)-Q(v')-Q(v)
$$
for all $v',v\in V$.

From the classification of quadratic forms over finite fields, we can find a basis of $V$ in which $Q$ has the form

$$
Q(x_1,\dots,x_{2m+1})=x_1x_{2m+1}+\cdots + x_mx_{m+2}+cx_{m+1}^2
$$
where $c\in\F_q^\times$ is well-determined by $Q$ (up to a multiplication by a square).

\bigskip

Consider the groupo\"id

$$
X=\left[\{v\in V\,|\, Q(v)=0\}/\F_q^\times\right].
$$
Define the kernel 

$$
K:X\times X\rightarrow\Q, \hspace{1cm}([v'],[v])\mapsto \sum_{s\in\F_q^\times}\psi(B(sv',v))=\sum_{s\in\F_q^\times}\psi(B(v',sv))
$$
and the corresponding operator $\calF^X:\calC(X)\rightarrow\calC(X)$  by

$$
\calF^X(f)([v'])=\sum_{[v]\in X}\frac{1}{|{\rm Aut}([v])|}K([v'],[v])f([v])
$$
for all $f\in\calC(X)$.

The aim of this section is to prove the following result.

\begin{theorem}The Fourier operator $\calF^X$ is involutive, i.e.

$$
\calF^X\circ\calF^X(f)=q^{2m}\, f.
$$
\label{theo2}\end{theorem}

\begin{remark}The theorem is not true anymore if the dimension of $V$ is even.
\end{remark}
\bigskip

To prove Theorem \ref{theo2}, we first evaluate the sum

$$
f(v')=\sum_{v\in V, \,Q(v)=0}\psi(B(v',v))
$$
or equivalently the sum

$$
f(v')=\frac{1}{q}\sum_{v\in V, \lambda\in\F_q}\psi\left(\lambda Q(v)+B(v',v)\right).
$$
We have

\begin{align*}
f(v')&=\frac{1}{q}\sum_{v\in V}\psi(B(v',v))+\frac{1}{q}\sum_{v\in V,\,\lambda\in\F_q^\times}\psi\left(\lambda(Q(v)+B(v'/\lambda,v))\right)\\
&=q^{2m}\delta_0(v')+\frac{1}{q}\sum_{v\in V,\, \lambda\in\F_q^\times}\psi\left(\lambda\left(Q\left(\frac{v'}{\lambda}+v\right)-Q\left(\frac{v'}{\lambda}\right)\right)\right)\\
&=q^{n-1}\delta_0(v')+\frac{1}{q}\sum_{\lambda_1\in\F_q^\times}\psi(-\lambda_1Q(v'))\sum_{v_1\in V}\psi(\lambda_1^{-1}Q(v_1))
\end{align*}
by the variables change $v_1=\frac{v'}{\lambda}+v$ and $\lambda_1=1/\lambda$.

A straightforward calculation shows that

$$
\sum_{v\in V}\psi(\lambda Q(v))=q^m\alpha_o(c\lambda) S(\alpha_o,\psi).
$$
Therefore, setting $\alpha_o(0)=0$ and noticing that

$$
\sum_{x\in\F_q}\psi(yx)\alpha_o(x)=\alpha_o(y) S(\alpha_o,\psi)
$$
we get

\begin{align*}
f(v')&=q^{2m}\delta_0(v')+q^{m-1} S(\alpha_o,\psi)\,\alpha_o(c)\sum_{\lambda_1\in\F_q}\psi(-\lambda_1 Q(v'))\alpha_o(\lambda_1)
\\
&=q^{2m}\delta_o(v')+q^{m-1}S(\alpha_o,\psi)^2\alpha_o(c)\alpha_o(-Q(v'))
\end{align*}
and using that $S(\alpha_o,\psi)^2=\alpha_o(-1)q$ we end up with

\begin{equation}
f(v')=q^{2m}\delta_o(v')+q^m\alpha_o(c)\alpha_o(Q(v')).
\label{tech}\end{equation}

We now prove Theorem \ref{theo2}.

\begin{proof}[Proof of Theorem \ref{theo2}] By Lemma \ref{involutive} we need to prove that for $[v],[v'']\in X$ we have

$$
\Delta([v''],[v])=\sum_{[v']\in X}\frac{1}{|{\rm Aut}([v'])|}K([v''],[v'])K([v'],[v])=\begin{cases}q^{2m}|{\rm Aut}([v])|&\text{ if }[v'']=[v],\\
0 &\text{ otherwise.}\end{cases}
$$
We have
\begin{align*}\Delta([v''],[v])&=\sum_{s,s''\in\F_q^\times}\sum_{[v']\in X}\frac{1}{|{\rm Aut}([v'])|}\psi(B(s''v'',v'))\psi(B(v',sv))\\
&=\sum_{s,s''}\sum_{[v']}\frac{1}{|{\rm Aut}([v'])|}\psi(B(sv+s''v'',v'))
\end{align*}
Notice that

\begin{align*}
\sum_{s\in\F_q^\times}\,\sum_{v',\,Q(v')=0}\psi(B(sX,v'))&=\sum_s\left(1+\sum_{[v']\neq 0}\sum_{\sigma\in\F_q^\times}\psi(B(sX,\sigma v'))\right)\\
&=(q-1)+\sum_{[v']\neq 0}\sum_\sigma\sum_s\psi(B(sX,\sigma v'))\\
&=(q-1)+\sum_{[v']\neq 0}\sum_\sigma\sum_s\psi(B(sX,v')\\
&=\sum_s\left(1+(q-1)\sum_{[v']\neq 0}\psi(B(sX,v')\right)\\
&=(q-1)\sum_s\sum_{[v']\in X}\frac{1}{|{\rm Aut}([v'])|}\psi(B(sX,v'))
\end{align*}
Therefore

\begin{align*}
\Delta([v''],[v])&=\sum_{s,s''}\sum_{[v']\in X}\frac{1}{|{\rm Aut}([v'])|}\psi(B(sv+s''v'',v'))\\
&=\frac{1}{q-1}\sum_{s,s''}\sum_{v', Q(v')=0}\psi(B(sv+s''v'',v'))
\end{align*}
From Formula (\ref{tech}) we get

$$
\Delta([v''],[v])=\frac{1}{q-1}\left(q^{2m}\sum_{s,s''}\delta_o(sv+s''v'')+q^m\alpha_o(c)\sum_{s,s''}\alpha_o(Q(sv+s''v''))\right)
$$
The first sum equals $(q-1)|{\rm Aut}([v])|$ if $[v]=[v'']$ and $0$ otherwise. Since

$$
Q(sv+s''v'')=B(sv,s''v'')+Q(sv)+Q(s''v'')=ss''B(v,v'')
$$
we get that 
$$
\sum_{s,s''}\alpha_o(Q(sv+s''v''))=\sum_{s,s''}\alpha_o(ss''B(v,v''))=0
$$
because $\alpha_o$ is a non-trivial character on $\F_q^\times$ and $\alpha_o(0)=0$.
\bigskip

We deduce that

$$
\Delta([v''],[v])=\begin{cases}q^{2m}|{\rm Aut}([v])|&\text{ if }[v]=[v''],\\
0&\text{ otherwise}.\end{cases}
$$
\end{proof}

\subsection{Proof of Theorem \ref{mainresult}}

\begin{nothing}\label{SL_2}Assume first that $G=\SL_2$. 

In this case we have

$$
\calG=\left[\{(x,\alpha)\in\gl_2\times\gl_1\,|\, \det(x)=\alpha^2\}/\GL_1\right]\times [\gl_1/\GL_1].
$$
Put

$$
\calG_1:=\left[\{(x,\alpha)\in\gl_2\times\gl_1\,|\, \det(x)=\alpha^2\}/\GL_1\right],\hspace{.5cm}\calG_2:=[\gl_1/\GL_1].
$$
Then

$$
\phi^\calG=\phi^{\calG_1}\boxtimes\phi^{\calG_2}
$$
where
$$
\phi^{\calG_1}([x,\alpha]):=\sum_{s\in\F_q^\times}\psi(s(\Tr(x)+2\alpha)),\hspace{1cm}\phi^{\calG_2}([z])=\sum_{s\in\F_q^\times}\psi(sz).
$$
Since the Fourier transform $\calF^{\calG_2}$ with respect to the kernel $\phi^{\calG_2}$ is clearly involutive, it remains to see that the Fourier transform $\calF^{\calG_1}$ with kernel $\phi^{\calG_1}$ is involutive.

Consider
$$
Q:(\gl_2\times\gl_1)^F\rightarrow\F_q^\times,\hspace{1cm}(x,\alpha)\mapsto \det(x)-\alpha^2.
$$
The associated bilinear form is given by

$$
B((y,\beta),(x,\alpha))=\Tr(\iota(y)x)-2\beta\alpha
$$
where $\iota:\gl_2\rightarrow \gl_2$ is the involution 

$$
\left(\begin{array}{cc}a&b\\c&d\end{array}\right)\mapsto\left(\begin{array}{cc}d&-b\\-c&a\end{array}\right)
$$
sending a matrix to the transpose of its co-factor matrix. 
\bigskip

By Theorem \ref{theo2}, the Fourier transform $\hat{\calF}^{\calG_1}:\calC(\calG_1^F)\rightarrow\calC(\calG_1^F)$ defined by

$$
\hat{\calF}^{\calG_1}(f)([y,\beta])=\sum_{[x,\alpha]}\frac{1}{|{\rm Aut}([x,\alpha])|}\phi^{\calG_1}([\iota(y)x,\alpha\beta])f([x,\alpha])
$$
is involutive from which we deduce the involutivity of $\calF^{\calG_1}$ as $\Tr(\iota(y)\iota(x))=\Tr(yx)$ for all $x,y\in\gl_2$.
\label{SL}\end{nothing}
\bigskip

\begin{nothing}\label{PGL_2}Let us now assume that $G=\PGL_2$ in which case we have

$$
\calG=\left[\left\{(x,\alpha,b)\in\gl_2\times\gl_1\times\gl_1\,|\, \det(x)=\alpha^2\right\}/\GL_1\times\GL_1\times\mu_2\right].
$$
If we put $X=\{(x,\alpha,b)\in\gl_2\times\gl_1\times\gl_1\,|\, \det(x)=\alpha^2\}$, then we get a $\mu_2$-torsor

$$
\tilde{f}:[X/\GL_1\times\GL_1]\longrightarrow[X/\GL_1\times\GL_1\times\mu_2]=\calG.
$$
which extends the natural quotient $f:\SL_2\rightarrow\PGL_2$ (notice that $[X/\GL_1\times\GL_1]$ is the stack $\calG$ for $G=\SL_2$). If we denote by $\phi$ (instead of $\phi^\calG$ to avoid any confusion) the kernel on $[X/\GL_1\times\GL_1]^F$ studied in the previous section \S \ref{SL} we get that

$$
\phi^\calG=\tilde{f}^F_!(\phi)
$$
extending the formula (\ref{f}). Therefore, by \S \ref{SL} and the result of \S\ref{Involutive} we get Theorem \ref{mainresult} for $G=\PGL_2$.
\end{nothing}

\subsection{Remarks on the torus case}\label{geounique}

Recall, see Formula (\ref{BKT}), that the BK-Fourier kernel on $T^F$ is given by

$$
\phi^T=\rho^F_!(\psi\circ\Tr)
$$
where $\rho:T_3\rightarrow T$ is given by

\begin{align*}
\rho(a,b,c)=(a/c,c/a)&\hspace{.5cm}\text{ if } G=\SL_2,\\
\rho(a,b,c)=[a^2b,bc^2]&\hspace{.5cm}\text{ if }G=\PGL_2.
\end{align*}

The following morphism induces a bijection between isomorphism classes of objects

$$
\tilde{\rho}:[\t_3/{\rm Ker}(\rho)]\rightarrow\calT, \hspace{1cm}[a,b,c]\mapsto [(a^2,c^2),ac,b]
$$
where 

\begin{align*}
&\calT=\left[\left\{(x,\alpha,b)\in\t_2\times\gl_1\times\gl_1\,|\, \det(x)=\alpha^2\right\}/\GL_1\times\GL_1\right]&\text{ if }G=\SL_2,\\&\calT=\left[\left\{(x,\alpha,b)\in\t_2\times\gl_1\times\gl_1\,|\, \det(x)=\alpha^2\right\}/\GL_1\times\mu_2\times\GL_1\right]&\text{ if }G=\PGL_2.\end{align*}
Moreover, the size of the stabilizers are the same on both sides of $\tilde{\rho}$. Therefore, since the kernel $\phi^{[\t_3/{\rm Ker}(\rho)]}$ from \S \ref{toricase}  defines an involutive Fourier transforms $\calF^{[\t_3/{\rm Ker}(\rho)]}$, the kernel 

$$
\phi^\calT:=\tilde{\rho}^F_!(\phi^{[\t_3/{\rm Ker}(\rho)]})
$$ 
defines also an involutive Fourier transforms $\calF^\calT$ on $\calC(\calT^F)$.
\bigskip

Notice that even if $\tilde{\rho}$ is not an isomorphism of stacks (as $[\t_3/{\rm Ker}(\rho)]$ is smooth and $\calT$ is singular), both embeddings 

$$
T\simeq T_3/{\rm Ker}(\rho)\hookrightarrow [\t_3/{\rm Ker}(\rho)],\hspace{1cm} T\hookrightarrow\calT
$$
provide a good space of functions with an involutive  Fourier transform on it that extends the BK-Fourier operator  $\calF^T=\calF^T_\rho:\calC(T^F)\rightarrow\calC(T^F)$. However, the first one does not have an analogue for the group $G$ while the second one does.
\bigskip

Following \cite{L}, we can extend the definition of Deligne-Lusztig induction $R_T^G$ and $R_{T_\sigma}^G$ to operators

$$
R_{\calT}^{\calG}:\calC(\calT^F)\rightarrow\calC(\calG^F), \hspace{1cm}R_{\calT_\sigma}^{\calG}:\calC(\calT_\sigma^F)\rightarrow\calC(\calG^F).
$$
Then we can verify that

\begin{equation}
\phi^\calG=\frac{1}{2}\left(R_\calT^\calG(\phi^\calT)+R_{\calT_\sigma}^\calG(\phi^{\calT_\sigma})\right)
\label{extBK}\end{equation}
where $\phi^{\calT_\sigma}$ is obtained by pushing $\psi\circ\Tr$ along the map $\t_3^{\sigma F}\rightarrow \calT^{\sigma F}$.

\section{Partial extension of the BK-Fourier kernel in the case $G=\GL_2$}

In this section $q$ is assumed to be odd as before. We construct a stack $\calG_\rho$ and an extension of BK-kernel in the $G=\GL_2$ case by analogy with the construction in the $\SL_2$ and $\PGL_2$ cases. But unfortunately, the resulting Fourier operator is not involutive and so our construction is only a partiel solution to the problem.

\subsection{The stack $\calG$}\label{G}

Assume that $G=\GL_2$ with $\rho^\vee={\rm Sym}^2$. By analogy with the $\SL_2$ and  $\PGL_2$ cases, the results of \S \ref{GL} suggest to consider the stack
$$
\calG:=\left.\left[\left\{(x,\alpha,b)\in\gl_2\times\gl_1\times\gl_1\,|\, \det(x)=\alpha^2\right\}\right/\GL_1\times\mu_2\right]
$$
where $\GL_1$ acts as $t\cdot[x,\alpha,b]=[tx,t\alpha,t^{-1}b]$ and $\mu_2$ acts on the second coordinate as $\alpha\mapsto -\alpha$. 

Notice that the open substack of $[x,\alpha,b]$ with $\alpha,b\neq 0$ is isomorphic to $\GL_2$ via the projection 
$$
(x,\alpha,b)\mapsto bx.
$$
We thus have a $G\times G$-equivariant embedding

$$
G\hookrightarrow\calG
$$
for the right and left multiplication.
\bigskip

Consider the toric analogue of $\calG$

$$
\calT:=\left.\left[\left\{(x,\alpha,b)\in\t_2\times\gl_1\times\gl_1\,|\, \det(x)=\alpha^2\right\}\right/\GL_1\times\mu_2\right]
$$
Then $\rho:T_3\rightarrow T, (a,b,c)\mapsto (a^2b,bc^2)$ extends to a morphism of stacks

\begin{equation}
\tilde{\rho}:[\t_3/{\rm Ker}(\rho)]\rightarrow\calT, \hspace{1cm}[a,b,c]\mapsto [(a^2,c^2),ac,b].
\label{GLmor}
\end{equation}
Notice that $\tilde{\rho}$ is an isomorphism on the open subset $\{(a,c)\neq(0,0)\}$ but not globally (as for instance $\calT$ is singular). 

Above $\{[0,b,0]\}$ with $b\neq 0$, the morphism becomes

$$
\left([\{b\neq 0\}/\GL_1]\rightarrow[\{b\neq 0\}/\GL_1]={\rm Spec}(\overline{\F}_q)\right)\times \B(\mu_2)
$$
where $s\in\GL_1$ acts by $s^2$ on the source and by $s$ on the target.

Therefore the morphism of stacks (\ref{GLmor})  is not a bijection between isomorphism classes of objects (unlike the cases of $\SL_2$ and $\PGL_2$).

\subsection{The BK-Fourier kernels $\phi^\calT$, $\phi^{\calT_\sigma}$}

Recall that

$$
{\rm Ker}(\rho)=\{(s,s^{-2},\varepsilon s)\,|\, \varepsilon\in\mu_2,\, s\in\GL_1\}.
$$
Let  $q:\t_3\rightarrow[\t_3/{\rm Ker}(\rho)]$ be the quotient map.

Consider on $[\t_3/{\rm Ker}(\rho)]^F$ and on  $[\t_3/{\rm Ker}(\rho)]^{\sigma F}$ the kernels 

$$
\phi^F=q^F_!(\psi\circ\Tr),\hspace{1cm}\phi^{\sigma F}=q^{\sigma F}_!(\psi\circ\Tr)
$$
i.e. 

$$
\phi^F([a,b,c])=\sum_{s\in\F_q^{\times},\varepsilon\in\mu_2}\psi\left(s(a+\varepsilon c)+s^{-2}b\right),\hspace{.5cm}\phi^{\sigma F}([a',b,a'{^q}])=\sum_{s^q=\pm s}\psi\left(s(a'+\varepsilon a'{^q})+bs^{-2}\right)
$$
considered in \S \ref{toricase}.

Analogously to the cases $G=\SL_2,\PGL_2$ (see \S \ref{geounique}), we consider the kernels 

$$
\phi^\calT:=-(\tilde{\rho}{^F})_!(\phi^F),\hspace{.5cm}\text{ and }\hspace{.5cm}\phi^{\calT_\sigma}:=-(\tilde{\rho}{^{\sigma F}})_!(\phi^{\sigma F}).
$$
They are given by the formulas

\begin{align*}
\phi^\calT([(a^2,c^2),ac,b])&=-\sum_{s\in\F_q^\times,\varepsilon\in\mu_2}\psi\left(s(a+\varepsilon c)+s^{-2}b\right)\\
\phi^{\calT_\sigma}([a'{^2},a'{^{2q}},a'{^{q+1}},b])&=-\sum_{s^q=\pm s}\psi\left(s(a'+\varepsilon a'{^q})+bs^{-2}\right)
\end{align*}
By operating the variable change $s(a+\varepsilon c)\leftrightarrow s'$ when $a+\varepsilon c\neq 0$ we get

$$
\phi^\calT([(a^2,c^2),ac,b])=\begin{cases}-\sum_{s,\varepsilon}\psi\left(s+s^{-2}b(a^2+c^2+2\varepsilon ac)\right)&\text{ if }a^2\neq c^2\\
-\kappa(b)-\sum_s\psi(s+4s^{-2}ba^2)&\text{ if }a=\varepsilon c\neq 0\\
-2\kappa(b)&\text{ if }a=c=0\end{cases}
$$
We thus get the values on the elements $[x,\alpha,b]$ as follows

\begin{scriptsize}
\begin{equation}
\label{tableGL}
\begin{array}{|c|c|c|c|}
\hline
&&&\\
&\left[\left(\begin{array}{cc}a^2&0\\0&a^2\end{array}\right),a^2,b\right]&\left[\left(\begin{array}{cc}a^2&0\\0&c^2\end{array}\right),ac,b\right] &  \left[\left(\begin{array}{cc}a'{^2}&0\\0& a'{^{2q}}\end{array}\right), a'{^{q+1}},b\right]\\
&a,b\in\F_q&a,c\in\F_q, a^2\neq c^2&a'{^2}\neq a'{^{2q}}\\
\hline
&&&\\
\phi^\calT&\begin{cases}-\kappa(b)-1-\Phi([x,\alpha,b])&\text{ if } a\neq 0\\-2\kappa(b)&\text{ if }a=0\end{cases}&-\Phi([x,\alpha,b])&\times \\
&&&\\
\hline
&&&\\
\phi^{\calT_\sigma}&\begin{cases}-\kappa'(b)-1-\Phi([x,\alpha,b])&\text{ if } a\neq 0\\-2\kappa'(b)&\text{ if }a=0\end{cases}&\times&-\Phi([x,\alpha,b])\\
&&&\\
\hline
\end{array}
\end{equation}
\end{scriptsize}
where for $[x,\alpha,b]\in\calG^F$ we put

$$
\Phi([x,\alpha,b]):=\sum_{s\in\F_q^\times,\,\varepsilon\in\mu_2}\psi\left(s+s^{-2}b(\Tr(x)+2 \varepsilon \alpha)\right).
$$
\begin{remark}(Non-involutivity) The Fourier operator $\calF^\calT$ defined from $\phi^\calT$  is not involutive. This follows from Lemma \ref{involutive} as a direct calculation shows that

$$
\Delta([(1,0),0,1],[(0,1),0,1])=4\alpha_o(-1) q\neq 0.
$$

\label{non-involutive}\end{remark}

%
%
%
%
%
%
%
%
%
%
%
%

\subsection{The BK-Fourier kernel $\phi^\calG$}\label{phiG}

We define the BK-Fourier kernel $\phi^\calG$ on $\calG^F$ by the formula (\ref{extBK}).
\bigskip

We compute the values of the BK-Fourier kernel $\phi^\calG$ on $\calG^F$: we have  $\phi^\calG([x,\alpha,b])=0$ if $\alpha\in\gl_1^{F'}$ and otherwise it is given by the following table

\begin{scriptsize}
\begin{equation}
\label{table}
\begin{array}{|c|c|c|c|c|c|}
\hline
&&&&\\
&\left[\left(\begin{array}{cc}a^2&0\\0&a^2\end{array}\right),a^2,b\right]&\left[\left(\begin{array}{cc}a^2&0\\0&c^2\end{array}\right),ac,b\right] &  \left[\left(\begin{array}{cc}x^2&0\\0& x^{2q}\end{array}\right), x^{q+1},b\right]&\left[\left(\begin{array}{cc}a^2&1\\0&a^2\end{array}\right),a^2,b\right]\\
&a,b\in\F_q&a,b,c\in\F_q, a^2\neq c^2&x^2\neq x^{2q}, b\in\F_q&a,b\in\F_q\\
\hline
&&&&\\
\phi^{\mathcal{G}}&\begin{cases}-\left(qS(\alpha_o,\psi_b)\delta_{b\neq 0}+q\delta_{b=0}+\Phi([x,\alpha,b])\right)&\text{ if }a\neq 0\\-2(qS(\alpha_o,\psi_b)\delta_{b\neq 0}+q\delta_{b=0}-1)&\text{ if }a=0\end{cases}&-\Phi([x,\alpha,b])&-\Phi([x,\alpha,b])&\begin{cases}-(q\delta_{b=0}+\Phi([x,\alpha,b]))&\text{ if }a\neq 0\\-2(q\delta_{b=0}-1)&\text{ if }a=0\end{cases} \\
&&&&\\
\hline
\end{array}
\end{equation}
\end{scriptsize}
We can see by  a direct calculation that the operator  $\calF^\calG$ on $\calC(\calG^F)$ defined from $\phi^\calG$ is not involutive (see Remark \ref{non-involutive}).

%
%

\section{The characteristic $2$ case}\label{char2}

In this section we assume that $\F_q$ is a finite field of characteristic $2$ (every element of $\F_q$ is then a square in $\F_q$).  In this case, our approach gives a complete answer to the problem in the three cases $\SL_2$, $\PGL_2$ and $\GL_2$. We explain it when $G=\GL_2$. 

\subsection{BK-Fourier operator on $G$}

First of all, notice that since the characteristic is $2$ we have

$$
{\rm Ker}(\rho^F)={\rm Ker}(\rho^{\sigma F})=\{(s,s^{-2},s)\,|\, s\in\F_q^\times\}.
$$
Following the calculation in \S \ref{GL} we find

\begin{align*}
\phi^T([(a^2,c^2),ac,b])&=\begin{cases}-\sum_{s\in\F_q^\times}\psi\left(s+s^{-2}(a^2+c^2)b\right)& \text{ if }a^2\neq c^2 \text{ i.e.}\, a\neq c,\\
-\kappa(b) &\text{ if }a^2=c^2.\end{cases}
\end{align*}
and

\begin{align*}
\phi^{T_\sigma}([(a'{^2},a'{^{2q}}),a'{^{q+1}},b])&=\begin{cases}-\sum_{s\in\F_q^\times}\psi\left(s+s^{-2}(a'{^2}+a'{^{2q}})b\right)& \text{ if }a'{^2}\neq a'{^{2q}} \text{ i.e.}\, a'{^q}\neq a',\\
-\kappa(b) &\text{ otherwise}.\end{cases}
\end{align*}

Since the characteristic is $2$, the identification (\ref{idGL}) becomes

 $$
 G\simeq \{(x,\alpha,b)\in\GL_2\times\GL_1\times\GL_1\,|\,\det(x)=\alpha^2\}/\GL_1
 $$

Writing the elements of $G$ as $[x,\alpha,b]$ we get the values of the BK-Fourier kernel as follows

\begin{scriptsize}
\begin{equation}
\label{BK2}
\begin{array}{|c|c|c|c|c|c|}
\hline
&&&&\\
&\left[\left(\begin{array}{cc}a^2&0\\0&a^2\end{array}\right),a^2,b\right]&\left[\left(\begin{array}{cc}a^2&0\\0&c^2\end{array}\right),ac,b\right] &  \left[\left(\begin{array}{cc}x^2&0\\0& x^{2q}\end{array}\right), x^{q+1},b\right]&\left[\left(\begin{array}{cc}a^2&1\\0&a^2\end{array}\right),a^2,b\right]\\
&a,b\in\F_q&a,b,c\in\F_q, a^2\neq c^2&x^2\neq x^{2q}, b\in\F_q&a,b\in\F_q\\
\hline
&&&&\\
\phi^G&-\kappa(b)&-\sum_{s\in\F_q^\times}\psi\left(s+s^{-2}b\Tr(x)\right)&-\sum_{s\in\F_q^\times}\psi\left(s+s^{-2}b\Tr(x)\right)&-\kappa(b)\\
&&&&\\
\hline
\end{array}
\end{equation}
\end{scriptsize}

\subsection{Extending the BK-Fourier operator}

Since the characteristic is $2$, the morphism 

$$
\rho:T_3\rightarrow T,\hspace{.5cm}(a,b,c)\mapsto (a^2b,bc^2)
$$
extends to a group morphism

$$
\rho:\GL_2\times\GL_1\rightarrow\GL_2,\hspace{.5cm}(x',b)\mapsto \iota_2(x')b
$$
where $\iota_2:\gl_2\rightarrow\gl_2$ is the algebra isomorphism defined by
$$
\iota_2\left(\begin{array}{cc}a&b\\c&d\end{array}\right):=\left(\begin{array}{cc}a^2&b^2\\c^2&d^2\end{array}\right).
$$

\begin{remark}Notice that 

$$
\phi^G=-\rho_!(\psi\circ\Tr)
$$
where $\Tr(x',b)=\Tr(x')+b$ for $(x',b)\in\gl_2\times\gl_1$.

\end{remark}

The  kernel of $\rho$ is

$$
{\rm Ker}(\rho)=\left.\left\{\left(\left(\begin{array}{cc}s&0\\0&s\end{array}\right),s^{-2}\right)\,\right| s\in\GL_1\right\}.
$$
We consider the $G\times G$-equivariant open embedding

\begin{equation}
G\simeq (\GL_2\times\GL_1)/{\rm Ker}(\rho)\hookrightarrow \calG:=\left[(\gl_2\times\gl_1)/{\rm Ker}(\rho)\right].
\label{emb}\end{equation}
We then consider on $\calG^F$ the kernel

$$
\phi^\calG([x',b])=-\sum_{s\in\F_q^\times}\psi\left(s\Tr(x')+s^{-2}b\right)
$$
Consider the two-variable kernel $K^\calG:\calG^F\times\calG^F\rightarrow\Q$ defined by

$$
K^\calG([x'_1,b_1],[x'_2,b_2])=\phi^\calG([x'_1x'_2,b_1b_2]).
$$

\begin{theorem} (1) The Fourier operator $\calF^\calG:\calC(\calG^F)\rightarrow\calC(\calG^F)$ defined from $K^\calG$ is involutive, i.e.

$$
\calF^\calG\circ\calF^\calG(f)=q^5\, f,
$$
for all $f\in\calC(\calG^F)$.

\noindent (2) $\calF^\calG$ extends the BK-Fourier operator $\calF^G$.

\end{theorem}

\begin{proof} The first assertion follows from Lemma \ref{connected} and the fact that the Fourier transform on $\gl_2\times\gl_1$ defined from the kernel 

$$
(x',b)\mapsto \psi(\Tr(x')+b)
$$
is involutive.
For the second assertion we need to see that the restriction of $\phi^\calG$ to $G$ along the embedding (\ref{emb}) is the kernel $\phi^G$ whose values are given by the Table (\ref{BK2}). To see that we use (when $\Tr(x')\neq 0$) the variable change $s\Tr(x')\leftrightarrow s$ which allows us to pass from 

$$
\sum_s\psi(s\Tr(x')+s^{-2}b)\hspace{1cm}\text{to}\hspace{1cm}\sum_s\psi(s+s^{-2}b\Tr(x))
$$
with $x=\iota_2(x')$.

\end{proof}


\begin{thebibliography}{}

\bibitem{BK1}{\sc Braverman, A. {and} Kazhdan, D.}: {\em $\gamma$-functions of representations and lifting} (with an appendix by V. Vologodsky), GAFA 2000,  Geom. Funct. Anal. 2000, Special Volume, Part I, 237--278.


\bibitem{BK}{\sc Braverman, A. {and} Kazhdan, D.}: {\em $\gamma$-sheaves on reductive groups}, Studies in memory of {I}ssai {S}chur ({C}hevaleret/{R}ehovot,
              2000), Progr. Math. {\bf 210}, Birkh\"{a}user Boston, Boston, MA (2003), 22--47.
              
 \bibitem{C}{\sc Chen, T.-H.}: {\em On a conjecture of Braverman-Kazhdan}, J. Amer. Math. Soc. {\bf 35} (2022), 1171--1214.
              

\bibitem{NC}{\sc Cheng, S. {and} Ng\^o, B. C.}: {\em On a conjecture of Braverman and Kazhdan}, Int. Math. Res. Not. {\bf 20} (2018), 6177--6200.

\bibitem{GJ}{\sc Godement, R. {and} Jacquet, H.}: {\em Zeta functions of simple algebras}, Lecture Notes in Math. {\bf 260}, Springer-Verlag (1972).

\bibitem{La}{\sc Lafforgue, L.}: {\em Du transfert automorphe de Langlands aux formules de Poisson non lin\'eaires}, Ann. Inst. Fourier {\bf 66} (2016), 899--1012.

\bibitem{LL0}{\sc Laumon, G. and Letellier, E.}: {\em On the derived Lusztig correspondence}, Forum Math. Sigma {\bf 11} (2023).

\bibitem{LL}{\sc Laumon, G. and Letellier, E.}: {\em Note on a conjecture of Braverman-Kazhdan}, Adv. Math.  {\bf 419} (2023).

\bibitem{L}{\sc Letellier, E.}: {\em Fourier transforms of invariant functions on finite reductive Lie algebras}, \emph{Lecture Notes in Mathematics, 1859}. Springer-Verlag, Berlin, 2005. xii+165 pp. 

\bibitem{Ngo}{\sc Ng\^o, B. C.}: {\em Hankel transform, Langlands functoriality and functional equation of automorphic $L$-functions}, Jpn. J. Math. {\bf 15} (2020), 121--167.

\bibitem{Tate}{\sc Tate, J.}: {\em Fourier Analysis in Number Fields and Hecke's Zeta-Functions}, Ph.D. thesis, Princeton University, Princeton, NJ, 1950.

\end{thebibliography}
\end{document}